\documentclass[12pt,twoside]{amsart}
\usepackage{latexsym,amsmath,amsopn,amssymb,amsthm,amsfonts}
\usepackage[T2A]{fontenc}
\usepackage[cp1251]{inputenc}
\usepackage[russian,english]{babel}
\usepackage{graphicx}

\newtheorem{theorem}{Theorem}

\newtheorem{proposition}{Proposition}

\newtheorem{remark}{Remark}

\newtheorem{definition}{Definition}

\newtheorem{corollary}{Corollary}

\newtheorem{lemma}{Lemma}

\newcommand{\Lin}{\operatorname{Lin}}

\newcommand{\Ann}{\operatorname{Ann}}

\newcommand{\Sh}{\operatorname{sh}}
\newcommand{\Ch}{\operatorname{ch}}

\newcommand{\Trace}{\operatorname{trace}}
\newcommand{\Det}{\operatorname{det}}
\newcommand{\Const}{\operatorname{const}}
\newcommand{\ad}{\operatorname{Ad}}

\newcommand{\Th}{\operatorname{th}}
\newcommand{\Exp}{\operatorname{Exp}}
\newcommand{\trace}{\operatorname{trace}}
\newcommand{\Sim}{\operatorname{Sim}}
\newcommand{\Tg}{\operatorname{tg}}

\begin{document}
\begin{flushleft}
UDK 519.46 + 514.763 + 512.81 + 519.9 + 517.911
\end{flushleft}
\begin{flushleft}
MSC 22E30, 49J15, 53C17
\end{flushleft}

\title[group $SL(2)$]{Geodesics and shortest arcs of special sub-Riemannian metric on the Lie group  $SL(2)$}
\author{V.\,N.\,Berestovskii, I.\,A.\,Zubareva,}
\thanks{The work is partially supported by the Russian Foundation for Basic Research (Grant 14-01-00068-a), a grant of the Government of the Russian federation for the State Support of Scientific Research (Agreement \No 14.B25.31.0029), and State Maintenance Program for the Leading Scientific Schools of the Russian Federation (Grant NSh-2263.2014.1)}
\address{V.N.Berestovskii}
\address{The Sobolev Institute of Mathematics SD RAS, \newline 4 Acad. Koptyug avenue, 630090, Novosibirsk, Russia}
\email{vberestov@inbox.ru}
\address{I.A.Zubareva}
\address{The Sobolev Institute of Mathematics, Omsk Branch,  \newline 13 Pevtsova street, 644043, Omsk, Russia}
\email{i\_gribanova@mail.ru}
\maketitle
\maketitle {\small
\begin{quote}
\noindent{\sc Abstract.}
The authors found geodesics, shortest arcs, cut loci, and conjugate sets for left-invariant sub-Riemannian metric on the Lie group $SL(2)$, which is right-invariant relative to the Lie subgroup $SO(2)\subset SL(2)$ (in other words, for invariant sub-Riemannian metric on weakly symmetric space $(SL(2)\times SO(2))/SO(2)$).
\end{quote}}

{\small
\begin{quote}
\end{quote}}

\section*{Introduction}

In this paper, by means of general methods from \cite{Ber0},  we found geodesics, shortest arcs, cut loci and conjugate sets of left-invariant sub-Riemannian metric on the Lie group $SL(2)$ with condition that the metric is right-invariant relative to the  Lie subgroup $SO(2)\subset SL(2).$ A formula, analogous to (\ref{sol}), and statements of Theorem \ref{th} have been given with no proof in \cite{BR} with references to some sources; also there were proved statements of Theorem  \ref{cutloc}, but we apply other methods and prove in detail all results.

We got analogous results for special left-invariant sub-Riemannian metrics on the Lie groups $SO_0(2,1)=SL(2)/\{\pm e\}$ and $SO(3)\cong SU(2)/\{\pm e\}$ in \cite{Ber1} and \cite{BerZub}. In these papers, together with analogs of  (\ref{sol}) and Theorem \ref{th}, to find geodesics and shortest arcs we use their geometric interpretation as parallel unit vector fields along geodesics and isoperimetrices (solutions of Dido's problem, i.e. the curves of constant geodesic curvature) on the Lobachevskii plane $L^2$ and the unit Euclidean sphere $S^2$, as well as the Gauss--Bonnet theorems for $L^2$ and $S^2.$ In this paper, for this purpose we apply directly Theorem \ref{th}.

It is necessary to note especially that all sub-Riemannian manifolds under consideration in this paper and papers \cite{Ber1}, \cite{BerZub} are \textit{geodesic orbit}, i.e. every geodesic of such manifold is an orbit of  some 
1--parameter isometry group. This is closely connected to the fact that 
one can consider these manifolds as  \textit{weakly symmetric spaces} $(SL(2)\times SO(2))/SO(2),$ $(SO_0(2,1)\times SO(2))/SO(2),$ $(SO(3)\times SO(2))/SO(2)$ with invariant sub-Riemannian metric.
A.~Selberg introduced weakly symmetric spaces in paper \cite{Selb}, where he considered $(SL(2)\times SO(2))/SO(2)$ as unique (nonsymmetric) example of such spaces. O.S.~Yakimova gives in \cite{Yak} a classification of (simply connected) weakly symmetric Riemannian manifolds with a reductive isometry group; the third of above-mentioned spaces is given in line 8 of Table 1 of this paper. It is well known that every weakly symmetric Riemannian manifold with invariant Riemannian metric is geodesic orbit. Since any invariant sub-Riemannian metric on weakly symmertic space is a limit of a sequence of invariant Riemannian metrics, then it is true that \textit{weakly symmetric space with invariant sub-Riemannian metric with no abnormal geodesic is geodesic orbit}. Let us note in this regard that
\textit{symmetric spaces admit no invariant sub-Riemannian metric}
\cite{Ber2}, \cite{BGo}.

\section{Preliminaries}

The Lie group $GL(n)$ consists of all real $(n\times n)$--matrices $g=(g_{ij})$, $i,j=1,\dots,n$, such that  $\Det g\neq 0$, and the Lie subgroup $GL_0(n)$ (the connected component of the unit $e$ in $Gl(n)$) is defined by condition $\Det g>0$.  It is natural to consider both groups as open submanifolds in
$\mathbb{R}^{n^2}$ with coordinates $g_{ij}$, $i,j=1,\dots,n$.

Their Lie algebra $\mathfrak{gl}(n)=GL(n)_e:=GL_0(n)_e=\mathbb{R}^{n^2}$ is a vector space of all real $(n\times n)$--matrices with Lie bracket 
$$[a,b]=ab-ba;\quad a,b\in\mathfrak{gl}(n).$$

Let $e_{ij}\in \frak{gl}(n),$ $i,j=1,\dots n$, be a matrix which has 1 in $i$-th row and $j$-th column and 0 in all other places. $\Lin(a,b)$ denotes the linear span of vectors $a,\,b$. As an auxiliary tool we shall use the standard scalar product $(\cdot,\cdot)$ on the Lie algebra $\frak{gl}(n)=\mathbb{R}^{n^2}$ for $n=2$.

The Lie group $SL(n)\subset GL(n)$ of all real $(n\times n)$--matrices with the determinant 1 is a closed connected Lie subgroup of the Lie group $Gl_0(n)$  with the Lie algebra
\begin{equation}
\label{sl}
\mathfrak{sl}(n)=\{a\in\mathfrak{gl}(n)\,:\,\Trace(a)=\sum\limits_{l=1}^{n}a_{ll}=0\}.
\end{equation}

In case of left-invariant sub-Riemannian metric on a Lie group, every geodesic
is a left shift of some geodesic which starts at the unit. Thus later we shall consider only geodesics with unit origin. Theorem 5 in paper \cite{Ber0} implies the following theorem.

\begin{theorem}
\label{theor0}
Let $G$ be a connected Lie subgroup of the Lie group $GL(n)$
with the Lie algebra $\frak{g},$ $D$ is some totally nonholonomic left-invariant distribution on $G,$ a scalar product $\langle \cdot,\cdot\rangle$ on $D(e)$ is proportional to the restriction of the scalar product $(\cdot,\cdot)$ (to $D(e)$). Then every normal geodesic (i.e. locally shortest arc), parametrized by arclength, $\gamma=\gamma(t)$, $t\in (-a,a)\subset\mathbb{R}$, $\gamma(0)=e$, on $(G,d)$ with left-invariant sub-Riemannian metric $d$, defined by the distribution $D$ and the scalar product $\langle\cdot,\cdot\rangle$ on $D(e),$ satisfies the system of ordinary differential equations
\begin{equation}
\label{s1}
\dot{\gamma}(t)=\gamma(t)u(t),\quad u(t)\in D(e)\subset\frak{g},\quad \langle u(t),u(t)\rangle\equiv 1,
\end{equation}
\begin{equation}
\label{s2}
pr_{\frak{g}}([u(t)^{T},u(t)]+[u(t)^{T},v(t)])=\dot{u}(t)+\dot{v}(t),
\end{equation}
where $u=u(t)$, $v=v(t)\in\frak{g}$, $(v(t),D(e))\equiv 0$, $t\in (-a,a)\subset\mathbb{R}$, are some real-analytic vector functions.
\end{theorem}

It follows from equations (\ref{s1}), (\ref{s2}) that

\begin{corollary}
\label{cor}
Every geodesic, parametrized by arclength, in $(G,d)$ is a part of unique geodesic $\gamma=\gamma(t),$ $t\in \mathbb{R},$
parametrized by arclength, in $(G,d).$
\end{corollary}

\section{The search of geodesics in $(SL(2),\delta)$}

We are interested in the Lie group $SL(2)$. In consequence of (\ref{sl}), matrices
\begin{equation}
\label{ppk}
p_1=\frac{1}{2}(e_{11}-e_{22}),\,\,p_2=\frac{1}{2}(e_{12}+e_{21}),\,\,k=\frac{1}{2}(e_{21}-e_{12})
\end{equation}
constitute a basis of the Lie algebra $\frak{sl}(2)$.

\begin{theorem}
\label{theor1}
Let be given the basis (\ref{ppk}) of the Lie algebra  $\frak{sl}(2)$, $D(e)=\Lin(p_1,p_2),$ and scalar product $\langle\cdot,\cdot\rangle$ on $D(e)$
with orthonormal basis $p_1$, $p_2$. Then left-invariant distribution $D$ on the Lie group $SL(2)$ with given $D(e)$ is totally nonholonomic and the pair
$(D(e),\langle\cdot,\cdot\rangle )$ defines left-invariant sub-Riemannian metric $\delta$ on $SL(2)$. Moreover, any  geodesic $\gamma=\gamma(t),$ $t\in \mathbb{R},$ parametrized by arclength, in $(SL(2), \delta)$ with condition 
$\gamma(0)=e$ is a product of two 1-parameter subgroups:
\begin{equation}
\label{sol}
\gamma(t)=\gamma(\beta,\phi; t)=\exp(t(\cos\phi p_1 + \sin \phi p_2 +\beta k)) \exp(-t\beta k),
\end{equation}
where $\phi,$ $\beta$ are some arbitrary constants.
\end{theorem}

\begin{proof}
It follows from (\ref{ppk}) that
$$[p_1,p_2]=-k,\quad [p_1,k]= -p_2,\quad [p_2,k]=p_1.$$
This implies the first statement of Theorem~\ref{theor1}.

It is clear that on $D(e)$
$$\langle\cdot,\cdot\rangle=\frac{1}{2}(\cdot,\cdot).$$
In consequence of Theorem 3 in \cite{Ber0}, every geodesic on 3-dimensional Lie group with left-invariant sub-Riemannian metric is normal. Then it follows from Theorem \ref{theor0}  that one can apply ODE (\ref{s1}), (\ref{s2})  to find geodesics   $\gamma=\gamma(t)$, $t\in\mathbb{R}$, in $(SL(2),\delta)$.

It is clear that
\begin{equation}
\label{uv}
u(t)=\cos{\phi(t)}p_1+\sin{\phi(t)}p_2,\quad v(t)=-\beta(t)k,
\end{equation}
and the identity (\ref{s2}) can be written in the form
$$[\cos{\phi(t)}p_1+\sin{\phi(t)}p_2,-\beta(t)k]=\dot{\phi}(t)(-\sin{\phi(t)}p_1+\cos{\phi(t)}p_2)-\dot{\beta}(t)k.$$
In consequence of (\ref{ppk}), the expression on the left hand side of equality is equal to
$$\beta(t)(\cos{\phi(t)}p_2-\sin{\phi(t)}p_1).$$
We get identities $\dot{\beta}(t)=0$, $\dot{\phi}(t)=\beta(t)$. Hence
\begin{equation}
\label{betaphi}
\beta=\beta(t)=\Const,\quad \phi(t)=\phi+\beta t.
\end{equation}

In view of (\ref{s1}), (\ref{uv}), and (\ref{betaphi}), it must be
\begin{equation}
\label{b}
\dot{\gamma}(t)=\gamma(t)(\cos{(\beta t+\phi)}p_1+\sin{(\beta t+\phi)}p_2).
\end{equation}
Let us prove that (\ref{sol}) is a solution of ODE (\ref{b}). Really,
$$\exp(-t\beta k)=\left(\begin{array}{cc}
\cos{\frac{\beta t}{2}} & \sin{\frac{\beta t}{2}} \\
-\sin{\frac{\beta t}{2}} & \cos{\frac{\beta t}{2}}
\end{array}\right).$$

Then
$$\dot{\gamma}(t)=\exp(t(\cos\phi p_1 + \sin \phi p_2 +\beta k))(\cos\phi p_1 + \sin \phi p_2 +\beta k)\exp(-t\beta k)+$$
$$\gamma(t)(-\beta k)=\gamma(t)\exp(t\beta k)(\cos\phi p_1 + \sin \phi p_2 +\beta k)\exp(-t\beta k)+\gamma(t)(-\beta k)=$$
$$\gamma(t)\exp(t\beta k)(\cos\phi p_1 + \sin \phi p_2)\exp(-t\beta k)+\gamma(t)(\beta k)+\gamma(t)(-\beta k)=$$
$$\gamma(t)\left(\begin{array}{cc}
\cos{\frac{\beta t}{2}} & -\sin{\frac{\beta t}{2}} \\
\sin{\frac{\beta t}{2}} & \cos{\frac{\beta t}{2}}
\end{array}\right)
\left(\begin{array}{cc}
\frac{1}{2}\cos{\phi} & \frac{1}{2}\sin{\phi} \\
\frac{1}{2}\sin{\phi} & -\frac{1}{2}\cos{\phi}
\end{array}\right)
\left(\begin{array}{cc}
\cos{\frac{\beta t}{2}} & \sin{\frac{\beta t}{2}} \\
-\sin{\frac{\beta t}{2}} & \cos{\frac{\beta t}{2}}
\end{array}\right)=$$
$$\gamma(t)(\cos{(\beta t+\phi)}p_1+\sin{(\beta t+\phi)}p_2)=\gamma(t)u(t).$$
\end{proof}

\begin{remark}
\label{rem1}
Both 1-parameter subgroups in (\ref{sol}) are nowhere tangent to the distribution $D$ for $\beta\neq 0$ so any their interval has infinite length in the metric $\delta.$
\end{remark}

\begin{remark}
\label{rem2}
To change a sign of $\beta$ in (\ref{sol}) is the same as to change a sign of $t$
and to change angle $\phi_0$ by angle $\phi_0\pm \pi.$
\end{remark}

\begin{proposition}
\label{rem3}
For any matrix $B\in SO(2)=\exp(\mathbb{R}k)$, the map $l_{B}\circ r_{B^{-1}}$, where $l_{B}$ is the multiplication on the left by $B$, $r_{B^{-1}}$ is the multiplication on the right by $B^{-1}$, is simultaneously automorphism $\ad B$ of the Lie algebra $(\frak{sl}(2),[\cdot,\cdot])$, preserving
$\langle\cdot,\cdot\rangle$, and automorphism $I(B)$ of the Lie group $SL(2)$, preserving the distribution $D$ and the metric $\delta$. In particular,
\begin{equation}
\label{ad}
\ad B(p_1+\beta k)=\cos{\phi}p_1+\sin{\phi}p_2+\beta k
\end{equation}
if
\begin{equation}
\label{bbb}
B=\exp(\phi k)=\left(\begin{array}{cc}
\cos{\frac{\phi}{2}} & -\sin{\frac{\phi}{2}} \\
\sin{\frac{\phi}{2}} & \cos{\frac{\phi}{2}}
\end{array}\right).
\end{equation}
Consequently, the metric $\delta$ on $SL(2)$ is invariant under right shifts by elements of the subgroup 
$SO(2)\subset SL(2).$
\end{proposition}

\begin{proposition}
\label{prop}
The space $(SL(2),\delta)$ is geodesic orbit, i.e., every geodesic in $(SL(2),\delta)$ is an orbit of  some  1--parameter isometry subgroup of the space
$(SL(2),\delta).$
\end{proposition}

\begin{proof}
In consequence of the left-invariance of the metric $\delta$, it is enough to prove the statement for a geodesic given by (\ref{sol}). By the last statement of Proposition \ref{rem3}, the maps
$$\Phi(s)=l_{\exp(s(\cos\phi p_1+\sin\phi p_2+\beta k))}\circ r_{\exp(-s\beta k)},\quad s\in \mathbb{R},$$
form  1-parameter motions subgroup of $(SL(2),\delta).$  Additionally, in consequence of the matrix exponent properties,
$\Phi(s)(\gamma(t))=\gamma(t+s).$
\end{proof}

\begin{corollary}
\label{eqv}
If $\gamma(t),$ $t_0\leq t \leq t_0+T,$ $t_0\in \mathbb{R},$ $T>0,$ is a shortest arc, then for any number
$t_1\in \mathbb{R}$, $\gamma(t),$ $t_1\leq t \leq t_1+T,$ is a shortest arc.
\end{corollary}

\begin{lemma}
\label{lem}
Let $x=(x_{ij})\in\frak{sl}(2)$,
$$\Det (x):=-x_{11}^2-x_{12}x_{21},\quad \alpha:=\sqrt{\mid\Det (x)\mid}.$$
Then
\begin{equation}
\label{al1}
\exp(x)=e+x\quad\mbox{if}\quad\Det (x)=0,
\end{equation}
\begin{equation}
\label{al2}
\exp(x)=\Ch\alpha\cdot e+\frac{\Sh\alpha}{\alpha}x\quad\mbox{if}\quad\Det (x)<0,
\end{equation}
\begin{equation}
\label{al3}
\exp(x)=\cos{\alpha}\cdot e+\frac{\sin\alpha}{\alpha}x\quad\mbox{if}\quad\Det (x)>0.
\end{equation}
\end{lemma}

\begin{proof}
Characteristic polynomial of the matrix $x$ is equal to
$$P(\lambda)=\mid x-\lambda e\mid=\left |\begin{array}{cc}
x_{11}-\lambda & x_{12} \\
x_{21} & -x_{11}-\lambda
\end{array}\right |=\lambda^2+\Det (x).$$

By the Hamilton--Cayley theorem (see p.~93 in \cite{Gant}),
the matrix $x$ is a root of the polynomial $P(\lambda),$ i.e.,  $x^2=-(\Det (x))e$.
This implies (\ref{al1}) and
$$x^{2n+1}=\alpha^{2n}x,\,\,x^{2n}=\alpha^{2n}e,\quad\mbox{if}\quad \Det (x)<0,\,\,n\geq 1,$$
$$x^{2n+1}=(-1)^n\alpha^{2n}x,\,\,x^{2n}=(-1)^n\alpha^{2n}e,\quad\mbox{if}\quad \Det (x)>0,\,\,n\geq 1.$$
Therefore for $\Det (x)<0$,
$$\exp(x)=e+\sum\limits_{n=1}^{\infty}\frac{x^n}{n!}=e\sum\limits_{n=0}^{\infty}\frac{\alpha^{2n}}{(2n)!}+
\frac{x}{\alpha}\sum\limits_{n=0}^{\infty}\frac{\alpha^{2n+1}}{(2n+1)!},$$
and (\ref{al2}) is fulfilled. Analogously, for $\Det (x)>0$,
$$\exp(x)=e+\sum\limits_{n=1}^{\infty}\frac{x^n}{n!}=e\sum\limits_{n=0}^{\infty}\frac{(-1)^n\alpha^{2n}}{(2n)!}+
\frac{x}{\alpha}\sum\limits_{n=0}^{\infty}\frac{(-1)^n\alpha^{2n+1}}{(2n+1)!},$$
and (\ref{al3}) is true.
\end{proof}

\begin{theorem}
\label{th}
Put
\begin{equation}
\label{mn1}
m=\frac{t}{2},\quad n=1,\quad\text{if}\quad\beta^2=1,
\end{equation}
\begin{equation}
\label{mn2}
m=\frac{\Sh\frac{t\sqrt{1-\beta^2}}{2}}{\sqrt{1-\beta^2}},\quad
n=\Ch\frac{t\sqrt{1-\beta^2}}{2},\quad\text{if}\quad\beta^2<1,
\end{equation}
\begin{equation}
\label{mn3}
m=\frac{\sin\frac{t\sqrt{\beta^2-1}}{2}}{\sqrt{\beta^2-1}},\quad
n=\cos\frac{t\sqrt{\beta^2-1}}{2},\quad\text{if}\quad\beta^2>1.
\end{equation}
Then the geodesic  $\gamma(t)=\gamma(\beta,\phi;t)$ of left-invariant sub-Riemannian metric  $\delta$ on $SL(2)$
(see \,theorem~\ref{theor1}) is equal to
\begin{equation}
\label{geod}
\left(\begin{array}{cc}
n\cos\frac{\beta t}{2}+m\left(\cos{\left(\frac{\beta t}{2}+\phi\right)}+\beta\sin\frac{\beta t}{2}\right) & n\sin\frac{\beta t}{2}+m\left(\sin{\left(\frac{\beta t}{2}+\phi\right)}-\beta\cos\frac{\beta t}{2}\right) \\
-n\sin\frac{\beta t}{2}+m\left(\sin{\left(\frac{\beta t}{2}+\phi\right)}+\beta\cos\frac{\beta t}{2}\right) & n\cos\frac{\beta t}{2}+m\left(-\cos{\left(\frac{\beta t}{2}+\phi\right)}+\beta\sin\frac{\beta t}{2}\right)
\end{array}\right).
\end{equation}
\end{theorem}

\begin{proof}
Let $\phi=0$. Then (\ref{sol}) takes the form
$$\gamma(t)\mid_{\phi=0}=\exp(t(p_1+\beta k))\exp(-t\beta k).$$
Using (\ref{ppk}) and Lemma \ref{lem}, we get
$$\exp(t(p_1+\beta k))=\exp\left(\frac{t}{2}\left(\begin{array}{cc}
 1 & -\beta \\
\beta & -1
\end{array}\right)\right)=$$
$$n\left(\begin{array}{cc}
 1 & 0 \\
0 & 1
\end{array}\right)+m\left(\begin{array}{cc}
 1 & -\beta \\
\beta & -1
\end{array}\right)=
\left(\begin{array}{cc}
 n+m & -\beta m \\
\beta m & n-m
\end{array}\right).$$
By (\ref{bbb}), matrices $B=\exp(\phi k)$ and $\exp(t\beta k)$ commute.
It follows from this, (\ref{sol}), and Proposition \ref{rem3} that
 $$\gamma(t)=B\cdot\gamma(t)\mid_{\phi=0}\cdot B^{-1}=B\exp(t(p_1+\beta k))B^{-1}\exp(-t\beta k)=$$
$$\left(\begin{array}{cc}
\cos{\frac{\phi}{2}} & -\sin{\frac{\phi}{2}} \\
\sin{\frac{\phi}{2}} & \cos{\frac{\phi}{2}}
\end{array}\right)\left(\begin{array}{cc}
 n+m & -\beta m \\
\beta m & n-m
\end{array}\right)
\left(\begin{array}{cc}
\cos{\left(\frac{\beta t}{2}+\frac{\phi}{2}\right)} & \sin{\left(\frac{\beta t}{2}+\frac{\phi}{2}\right)} \\
-\sin{\left(\frac{\beta t}{2}+\frac{\phi}{2}\right)} & \cos{\left(\frac{\beta t}{2}+\frac{\phi}{2}\right)}
\end{array}\right).$$
Calculation of the product of last two matrices finishes the proof of Theorem~\ref{th}.
\end{proof}

\begin{corollary}
If $\phi=0$ then in the notation (\ref{mn1}), (\ref{mn2}), and (\ref{mn3}),
\begin{equation}
\label{gamma0}
\gamma(t)=\left(\begin{array}{cc}
n\cos\frac{\beta t}{2}+m\left(\cos{\frac{\beta t}{2}}+\beta\sin\frac{\beta t}{2}\right) & n\sin\frac{\beta t}{2}+
m\left(\sin{\frac{\beta t}{2}}-\beta\cos\frac{\beta t}{2}\right) \\
-n\sin\frac{\beta t}{2}+m\left(\sin{\frac{\beta t}{2}}+\beta\cos\frac{\beta t}{2}\right) & n\cos\frac{\beta t}{2}+
m\left(-\cos{\frac{\beta t}{2}}+\beta\sin\frac{\beta t}{2}\right)
\end{array}\right).
\end{equation}
\end{corollary}

\section{The set of symmetric matrices in $SL(2)$}

The following proposition is proved by direct calculations.

\begin{proposition}
\label{trm}
The numbers $\trace(c)=c_{11}+c_{22}$,
\begin{equation}
\label{mc}
m(c):=\frac{\sqrt{(c_{11}-c_{22})^2+(c_{12}+c_{21})^2}}{2}
\end{equation}
and the symmetry property for $(2\times 2)$--matrix $c=(c_{ij})$ are invariant relative to the conjugation by matrices of the subgroup $SO(2).$
If $c_{12}=c_{21}$ then
\begin{equation}
\label{mdet}
\left(\frac{\trace(c)}{2}\right)^2= \det (c)+ (m(c))^2.
\end{equation}
\end{proposition}

Obviously,

\begin{proposition}
\label{so2}
$c\in SO(2)$ if and only if $c\in SL(2)$ and $m(c)=0.$
\end{proposition}

\begin{proposition}
\label{sim}
1) The set $\Sim$ of all symmetric matrices from $SL(2)$ has the form
$\Sim=\Sim^+ \cup \Sim^-,$ where
$$\Sim^+=\{c\in \Sim\, \mid\, \trace(c)>0\},\,\,\Sim^-=(-e)\Sim^+,$$
and
\begin{equation}
\label{simp}
\Sim^+=\{c\in \Sim\, \mid\, \trace(c) \geq 2\},
\end{equation}
\begin{equation}
\label{simm}
\Sim^-=(-e)\Sim^+=\{c\in \Sim\, \mid\, \trace(c) \leq -2\}.
\end{equation}

2) $c=\pm e$ if and only if $c\in \Sim$ and $m(c)=0.$

3) The sets $\Sim^+,$ $\Sim^-,$ $\Sim$ are invariant relative to the conjugation by matrices of the subgroup 
$SO(2)\subset SL(2).$

4) For every matrix $c\in \Sim^+,$ $c\neq e,$ 
\begin{equation}
\label{tm}
\trace(c)/2=\sqrt{1+m(c)^2}.
\end{equation}

5) $c\in \Sim^+$ and $c\neq e$ if and only if
\begin{equation}
\label{vs}
c=\left(\begin{array}{ccc}
\Ch a + \cos 2b\Sh a& \sin 2b\Sh a \\
\sin 2b\Sh a& \Ch a-\cos 2b\Sh a
\end{array}\right)=
\end{equation}
$$\left(\begin{array}{ccc}
 \cos b & -\sin b \\
 \sin b & \cos b
\end{array}\right)
\left(\begin{array}{ccc}
 \Ch a + \Sh a & 0 \\
 0 & \Ch a - \Sh a
\end{array}\right)
\left(\begin{array}{ccc}
 \cos b & \sin b \\
-\sin b & \cos b
\end{array}\right),$$
where
\begin{equation}
\label{data}
\Ch a = \trace(c)/2,\,\,\Sh a = m(c),\,\,\cos 2b=(c_{11}-c_{22})/2m(c),\,\,\sin 2b=c_{12}/m(c).
\end{equation}
\end{proposition}

\begin{proof}
1) If $c\in \Sim$ then $c_{11}c_{22}=1+c_{12}^2\geq 1 > 0,$ consequently,
$c\in \Sim^+$ or $c\in \Sim^-.$ If $c\in \Sim^+$ then
\begin{equation}
\label{trd}
\frac{\trace(c)}{2}=\frac{c_{11}+c_{22}}{2}\geq \sqrt{c_{11}c_{22}}=
\sqrt{1+c_{12}^2}\geq 1.
\end{equation}
By (\ref{trd}), we get (\ref{simp}) and (\ref{simm}).

The statements 2), 3), 4) follow respectively from Propositions \ref{so2}, \ref{trm}, (\ref{mdet}).

5) The sufficiency is obvious. The representation of any matrix
$c\in \Sim^+-\{e\}$ in the form (\ref{vs}) guarantees the statement 2) and formulae (\ref{simp}), (\ref{tm}), (\ref{data}).
\end{proof}

\begin{corollary}
\label{exist}
For every number $m_0\geq 0$, there exists a unique up to the conjugation by matrices of the subgroup $SO(2)\subset SL(2)$ matrix $c\in \Sim^+$ such that $m(c)=m_0.$
Additionally, (\ref{tm}) is valid.
\end{corollary}

It follows from Theorem \ref{theor1}, \ref{th}, and Proposition \ref{sim} that

\begin{corollary}
\label{op}
A matrix $c\in \Sim^+-\{e\}$ if and only if
$c= \gamma(0,\phi; t),$ where
\begin{equation}
\label{data1}
\Sh (t/2) = m(c),\quad \cos \phi=(c_{11}-c_{22})/2m(c),\quad
\sin \phi=c_{12}/m(c).
\end{equation}
As a consequence, $\Sim^+=\exp(D(e)).$
\end{corollary}

\section{Cut loci and conjugate sets in $(SL(2),\delta)$}

Unlike the Riemannian manifolds, the exponential map $\Exp_x,$ $x\in M,$
for a sub-Riemannian manifold $(M,d)$ with no abnormal geodesic (as in the case of $(SL(2),\delta)$) are defined not on $TM$ and $T_xM$ but on $D(x)\times \Ann(D(x))$, where $D$ is a distribution on $M$ involved in the definition of $d,$ and
$$\Ann(D(x))=\{\psi\in T^{\ast}_xM: \langle\psi,D(x)\rangle=0\},$$
see \cite{VG}.
Otherwise, the cut loci and conjugate sets for such sub-Riemannian manifolds are defined in the same way as for the Riemannian ones \cite{GKM}.

\begin{definition}
\label{cutl}
The cut locus $C(x)$ (respectively, (the first) conjugate set $S(x)$ ($S_1(x)$)) for a point $x$ in a sub-Riemannian manifold $M$ (with no abnormal geodesic) is the set of ends $y\in M$ of all shortest arcs joining the point $x$ with
the point $y$ and noncontinuable beyond $y$ (respectively, the image of the set of (the first) critical points (along geodesics with the origin $x$) of the map $\Exp_x$
with respect to $\Exp_x$).
\end{definition}

The main result of this section is

\begin{theorem}
\label{cutloc}
For every $g\in (SL(2),\delta)$, $C(g)=gC(e)$ and $S(g)=gS(e)$. Moreover
\begin{equation}
\label{cutloc0}
C(e)=K(e)\cup S_1(e),
\end{equation}
where
\begin{equation}
\label{cutloc1}
K(e)=\Sim^{-}=\left\{c \in SL(2)\,\mid\,c^T=c,\,\trace(c)\leq -2 \right\},
\end{equation}
\begin{equation}
\label{congloc}
S_1(e)=SO(2)-\{e\}.
\end{equation}
Also $K(e)$ is diffeomorphic to $\mathbb{R}^2,$\,\,$S_1(e)$ is diffeomorphic to $\mathbb{R},$\,\,$S_1(e)\cap K(e)=\{-e\}.$
\end{theorem}

\begin{proposition}
\label{bou0}
Every segment  $\gamma(t)=\gamma(0,\phi;t)$, $0\leq t\leq t_1$, is a shortest arc.
\end{proposition}

\begin{proof}
It is known that the Lie group $SL(2)/\{\pm e\}$ is isomorphic to the Lie group of all orientation-preserving isometries of the Lobachevskii plane with sectional curvature $-1$, and the last group is isomorphic to the Lie group $SO_0(2,1),$  the connected component of the unit in the Lorentz group $SO(2,1)$ (see, for example, \cite{Ber1}). By Theorem 1 from \cite{Ber1}, there exists a locally isomorphic epimorphism of the Lie groups
 $$L:SL(2) \rightarrow SL(2)/\{\pm e\}\cong SO_0(2,1)$$
such that, in terms of this paper and paper \cite{Ber1},
$$dL(e)(p_1)=a,\quad dL(e)(p_2)=b,\quad dL(e)(k)=c.$$
Therefore the map $L:(SL(2),\delta)\rightarrow (SO_0(2,1),d)$ is a submetry \cite{BG} preserving the lengths of curves.
Consequantly, $L(\gamma(t)),$ $t\in \mathbb{R}$, is a geodesic and  simultaneously 1--parameter subgroup in
$(SO_0(2,1),d),$ so by Lemma 2 from \cite{Ber1}, every its segment is a shortest arc. Then the same statement is true for $\gamma.$
\end{proof}

It follows directly from Corollary \ref{op} and Proposition \ref{bou0} that
\begin{proposition}
\label{simnot}
$C(e)\cap Sim^+ = \emptyset.$
\end{proposition}

The following proposition was proved in paper \cite{BerZub1} (see Corollary 1 in \cite{BerZub1}):

\begin{proposition}
\label{bz}
If two points in a three-dimensional Lie group $G$ with a left-invariant sub-Riemannian metric  
are joined by two different geodesics, parametrized by arclength, of equal length, then any of these geodesics 
either is not a shortest arc or is not a part of a longer shortest arc.
\end{proposition}

\begin{proposition}
\label{sss}
If a segment $\gamma(t)$, $0\leq t\leq t_0$, of the geodesic (\ref{sol}) is a shortest arc and $\beta^2>1$, then $t_0\leq\frac{2\pi}{\sqrt{\beta^2-1}}$.
\end{proposition}

\begin{proof}
In consequence of (\ref{mn3}), $m=0$ and $n=-1$ for $t=\frac{2\pi}{\sqrt{\beta^2-1}}$. Substituting these $m$ and $n$ into (\ref{geod}), we get that
\begin{equation}
\label{ss}
\gamma\left(\frac{2\pi}{\sqrt{\beta^2-1}}\right)=
\left(\begin{array}{cc}
-\cos\frac{\pi\beta}{\sqrt{\beta^2-1}} & -\sin\frac{\pi\beta}{\sqrt{\beta^2-1}} \\
\sin\frac{\pi\beta}{\sqrt{\beta^2-1}} & -\cos\frac{\pi\beta}{\sqrt{\beta^2-1}}
\end{array}\right)
\end{equation}
does not depend on $\phi.$
It remains to apply Proposition \ref{bz}.
\end{proof}

\begin{proposition}
\label{suf}
\begin{equation}
\label{subset}
K(e)\cup SO(2)-\{e\}\subset C(e).
\end{equation}
\end{proposition}

\begin{proof}
Denote by $c=(c_{ij})$ any matrix (\ref{geod}). One can easily see that
\begin{equation}
\label{sistem}
c_{11}+c_{22}=2n\cos\frac{\beta t}{2}+2\beta m\sin\frac{\beta t}{2},\quad
c_{12}-c_{21}=2n\sin\frac{\beta t}{2}-2\beta m\cos\frac{\beta t}{2},
\end{equation}
\begin{equation}
\label{sist}
c_{11}-c_{22}=2m\cos\left(\frac{\beta t}{2}+\phi\right),\quad
c_{12}+c_{21}=2m\sin\left(\frac{\beta t}{2}+\phi\right).
\end{equation}
Note that the system of equalities   (\ref{sistem}), (\ref{sist}) is equivalent to (\ref{geod}).

It follows from (\ref{sist}) and (\ref{mn1}), (\ref{mn2}), (\ref{mn3}), (\ref{mc}) that if
$\beta^2\leq 1$ and $t\geq 0$ or $\beta^2> 1$ and $0\leq t\leq \frac{2\pi}{\sqrt{\beta^2-1}}$ as in Proposition \ref{sss}, then
\begin{equation}
\label{m}
m = m(c).
\end{equation}

If $c\in SO(2)-\{e\}$ then $m(c)=0$ by Proposition \ref{so2}, and
since $(SL(2),\delta)$ is a locally compact complete space with inner metric then in consequence of the Cohn--Vossen theorem \cite{CF},  there exists a shortest arc $\gamma(\beta,\phi_0;t),$
$0\leq t \leq T,$ joining $e$ and $c.$ By (\ref{m}), it must be
$m(T)=m(c)=0.$ Then $\beta^2>1,$ $T=\frac{2\pi}{\sqrt{\beta^2-1}}$ on the ground of  (\ref{mn1}), (\ref{mn2}), (\ref{mn3}), so $\gamma(\beta,\phi;T)=\gamma(\beta,\phi_0;T)$ for all $\phi.$  Then $c\in S(e)\cap C(e).$

Let $c\in K(e).$ In consequence of the Cohn--Vossen theorem,  there exists a shortest arc
$\gamma(\beta,\phi_0;t),$\,\, $0\leq t \leq T,$\,\, joining $e$ and $c.$ Since  $\trace(c)\leq -2$ by 
(\ref{cutloc1}), then on the ground of Corollary \ref{op}, $\beta \neq 0$.
Now in consequence of the equality $c_{12}=c_{21}$ and (\ref{geod}), the evenness of functions
$n$ and $\cos,$ the oddness of functions $m$ и $\sin,$ 
$\gamma(\beta,\phi_0;T)= \gamma(\beta,\phi_1;-T),$ where
$\phi_1=\beta T + \phi_0 + \pi.$ Therefore it follows from Proposition \ref{bz} that $c\in C(e).$
\end{proof}

\begin{proposition}
\label{se}
$$S(e)=(S_1(e)=SO(2)-\{e\}) \quad\cup $$
$$\left\{\gamma(\beta,\phi;t)\mid \Tg\left(\frac{t\sqrt{\beta^2-1}}{2}\right)=\frac{t\sqrt{\beta^2-1}}{2},\,\,\beta^2>1,\,\, t\neq 0\right\};$$
$$C\cap S(e)=C\cap S_1(e)=SO(2)-\{e\}.$$
\end{proposition}

\begin{proof}
On the ground of (\ref{mn1}), (\ref{mn2}), (\ref{mn3}), we get
$$\gamma(\beta,\phi;t)=2m\left[\cos\left(\frac{\beta t}{2}+\phi\right)p_1+\sin\left(\frac{\beta t}{2}+\phi\right)p_2\right]+ $$
$$2\left(\beta m\cos\frac{\beta t}{2}-n\sin\frac{\beta t}{2}\right)k+
\left(\beta m\sin\frac{\beta t}{2}+n\cos\frac{\beta t}{2}\right)e;$$
$$m'_t=\frac{n}{2}, \quad n'_t=\frac{1-\beta^2}{2}m;$$
$$m'_{\beta}= \frac{\beta}{1-\beta^2}\left(m-\frac{tn}{2}\right),\quad n'_{\beta}=-\frac{\beta t}{m},\quad \beta^2\neq 1;$$
$$m'_{\beta}=n'_{\beta}= 0,\quad \beta^2 = 1.$$

Using these relations, one can easily compute that
$$\gamma'_{\phi}=2m\left[-\sin\left(\frac{\beta t}{2}+\phi\right)p_1+
\cos\left(\frac{\beta t}{2}+\phi\right)p_2\right],$$
$$\gamma'_{t}=\frac{\beta}{2}\gamma'_{\phi}+n\left[\cos\left(\frac{\beta t}{2}+\phi\right)p_1+
\sin\left(\frac{\beta t}{2}+\phi\right)p_2\right]+
m\left(-\sin\frac{\beta t}{2}k + \frac{1}{2}\cos\frac{\beta t}{2}e\right)$$
for all $\beta\in\mathbb{R};$\quad if $\beta^2\neq 1$ then
$$\gamma'_{\beta}=\frac{t}{2}\gamma'_{\phi}+$$
$$\frac{2}{1-\beta^2}
\left(m-\frac{tn}{2}\right)\left\{\beta\left[\cos
\left(\frac{\beta t}{2}+\phi\right)p_1+
\sin\left(\frac{\beta t}{2}+\phi\right)p_2\right]+
\cos\frac{\beta t}{2}k + \frac{1}{2}\sin\frac{\beta t}{2}e\right\};$$
if $\beta^2=1$ then
$$\gamma'_{\beta}=\frac{t}{2}\gamma'_{\phi}+\frac{t^2}{4}\left(-2\sin\frac{\beta t}{2}k+ \cos\frac{\beta t}{2}e\right).$$
Besides the value $t=0,$ we get critical values only for $m=0$ or
$2m-tn=0,$ when $m\neq 0$ and $\beta^2\neq 1.$ This and the proof of Proposition \ref{sss} inmply the disjunction of the union, the inequality $|t|>2\pi/\sqrt{\beta^2-1}$ for points of the second set of the union, and the first statement. Now the second statement follows from Propositions \ref{sss}, \ref{suf}.
\end{proof}

\begin{theorem}
\label{con}
If $c\in C(e)$ for $(SL(2),\delta)$ then $c\in S_1(e)$ or there exist
$\beta_i$, $\phi_i\in\mathbb{R}$, $i=1,2$, $T>0$, such that
\begin{equation}
\label{TTT}
c=\gamma(\beta_1,\phi_1;T)=\gamma(\beta_2,\phi_2;T),
\end{equation}
where $T=T(\beta_1,\phi_1)$ is the smallest positive number for which there exist
$\beta_2$, $\phi_2\in\mathbb{R}$ such that the equality (\ref{TTT}) holds and geodesics $\gamma_1=\gamma(\beta_1,\phi_1;t)$ and  $\gamma_2=\gamma(\beta_2,\phi_2;t)$ are different.
\end{theorem}

\begin{proof}
Assume that $c=\gamma(\beta_1,\phi_1;T)\in C(e)-S(e),$\,\,$T> 0.$  Then
for every $n\in \mathbb{N},$ geodesic segment $\gamma(\beta_1,\phi_1;t),$
$0\leq t \leq T+1/n,$ isn't a shortest arc, and by the inverse map theorem, the map
$\gamma: (\beta,\phi,t)\rightarrow \gamma(\beta,\phi;t)$ is a diffeomorphism
in some neighbourhood $U$ of the point $(\beta_1,\phi_1,T).$ By the Cohn--Vossen theorem \cite{CF},
for sufficiently large numbers $n$ there exists a shortest arc 
$\gamma_n(t):=\gamma(\beta_n,\phi_n;t),$ $0\leq t \leq T_n,$ where
$T-1/n\leq T_n < T+1/n$ and $(\beta_n,\phi_n + 2\pi l,T_n)\notin U$ for all
$l\in \mathbb{Z},$ joining points $e$ and $\gamma(\beta_1,\phi_1;T+1/n).$
By the same reason there exist a subsequence $\{n_j\}_{j\in \mathbb{N}}$ and $(\beta_2,\phi_2)$ such that shortest arcs
$\gamma_{n_j}(t):=\gamma(\beta_{n_j},\phi_{n_j};t),$ $0\leq t \leq T_{n_j},$ converge to shortest arc  
$\gamma(\beta_2,\phi_2;t),$ $0\leq t \leq T,$ joining $e$ and $\gamma(\beta_1,\phi_1;T),$\,\, moreover, geodesics 
$\gamma_1=\gamma(\beta_1,\phi_1;t)$ and $\gamma_2=\gamma(\beta_2,\phi_2;t)$ are different. If $c\in S(e)$ then
$c\in S_1(e)$ in consequence of Proposition \ref{se}.
\end{proof}

\begin{remark}
This theorem and its proof extend to any homogeneous (sub)--Riemannian manifold
(with no strictly abnormal geodesic) only by changing notation.
\end{remark}

\begin{proposition}
\label{kk1}
If $c\in C(e)$ then $c\in S_1(e)$ or there exist
$\phi_1$, $\phi_2'\in\mathbb{R}$, $\beta>0$ such that
\begin{equation}
\label{equat}
c=\gamma(\beta,\phi_1;T)=\gamma(\beta,\phi_2';-T),
\end{equation}
where $\gamma(t)$ is defined by (\ref{geod}) and
\begin{equation}
\label{zT}
T=\min\{t>0\mid \gamma(\beta,\phi_1;t)=\gamma(\beta,\phi_2';-t)\}.
\end{equation}
\end{proposition}

\begin{proof}
Suppose that $c\in C(e)-S_1(e).$ Theorem \ref{con}, (\ref{TTT}), and (\ref{m}) imply that
\begin{equation}
\label{q1}
\frac{2}{T} m(\beta_1,T)=\frac{2}{T}m(\beta_2,T),
\end{equation}
where $m=m(\beta,t)$ is defined by (\ref{mn1})--(\ref{mn3}). But
\begin{equation}
\label{q2}
\frac{2}{T}m(\beta,T)=
\left\{\begin{array}{c}
1,\quad\mbox{если}\quad\beta^2=1, \\
\frac{\Sh x}{x},\quad\mbox{if}\quad\beta^2<1\,\,\mbox{and}\,\,x=\frac{T\sqrt{1-\beta^2}}{2}, \\
\frac{\sin x}{x},\quad\mbox{if}\quad\beta^2>1\,\,\mbox{and}\,\,x=\frac{T\sqrt{\beta^2-1}}{2}. \\
\end{array}\right.
\end{equation}

One can easily see that the following lemmas are valid.

\begin{lemma}
\label{lem2}
Function $y=\frac{\Sh x}{x}$, $x>0,$ increases and its range is interval $(1,+\infty).$
\end{lemma}

\begin{lemma}
\label{lem3}
Function $y=\frac{\sin x}{x}$, defined on segment $\left(0,\pi\right]$ increases and its range is interval  
$\left[0,1\right)$.
\end{lemma}

It follows from (\ref{q2}), Lemmas \ref{lem2} and \ref{lem3} that (\ref{q1}) holds only in the case $\beta_1^2=\beta_2^2$.

Let $\beta_1=\beta_2=\beta$. Then on the ground of (\ref{sistem}), (\ref{sist}), the equality (\ref{TTT}) holds in the case $\phi_{2}=\phi_{1}+2\pi l$,
$l\in\mathbb{Z}$ (i.e. the corresponding geodesics coincide), or in the case $m=0$ (i.e. $\gamma(\beta,\phi_1;T)\in SO(2)-\{e\}= S(e)\cap C(e)$ in consequence of (\ref{m}) and Propositions \ref{so2}, \ref{se}).

If $\beta_2=-\beta_1\neq 0$ (see Proposition \ref{bou0}) then by Remark \ref{rem2}, the equality  (\ref{TTT}) is equivalent to the equality  (\ref{equat}).
\end{proof}

\begin{proposition}
\label{incl}
\begin{equation}
\label{sub}
C(e)-S_1(e)\subset K(e).
\end{equation}
\end{proposition}

\begin{proof}
Let $c\in C(e)-S_1(e)$. Since $c_{12}-c_{21}$ is an odd function relative to $t$ on the ground of (\ref{sistem}), 
then $c_{12}=c_{21}$ by Proposition \ref{kk1}. Now inclusion (\ref{sub}) follows from Propositions
\ref{simnot} and \ref{sim}.
\end{proof}

All statements  of Theorem \ref{cutloc}, except for the last one, follow from Propositions \ref{suf}, \ref{se}, \ref{kk1}, \ref{incl}; the last statement follows from (\ref{sol}), (\ref{simm}), Proposition \ref{bou0}, Corollary \ref{op}.

\section{Noncontinuable shortest arcs on $(SL(2),\delta)$}

The following theorem constitutes the main result of this section. 

\begin{theorem}
\label{mai}
Let $\beta\neq 0$ and $\gamma=\gamma(\beta, \phi; t)$, $0\leq t\leq T$, be a noncontinuable shortest arc (\ref{geod}). Then

1) If $\mid\beta\mid\geq\frac{2}{\sqrt{3}}$ then $T=\frac{2\pi}{\sqrt{\beta^2-1}}$.

2) If $\beta^2=1$ then $T\in (2\pi,3\pi)$ and $T$ satisfies the system of equations
$$\cos\frac{T}{2}=\frac{-1}{\sqrt{1+(T/2)^2}},\quad \sin\frac{T}{2}=\frac{-T/2}{\sqrt{1+(T/2)^2}}.$$

3) If $\beta^2<1$ then $T\in\left(\frac{2\pi}{\mid\beta\mid},\frac{3\pi}{\mid\beta\mid}\right)$ and $T$ satisfies the system of equations

$$\cos kx=\frac{-1}{\sqrt{1+k^2\Th^2x}},\quad \sin kx=\frac{-k \Th x}{\sqrt{1+k^2\Th^2x}},$$
where
\begin{equation}
\label{kx}
k=\frac{\mid\beta\mid}{\sqrt{1-\beta^2}},\quad x=\frac{T\sqrt{1-\beta^2}}{2}= \frac{T}{2\sqrt{1+k^2}}.
\end{equation}

4) If $\mid\beta\mid=\frac{3}{2\sqrt{2}}$ then $T=2\sqrt{2}\pi$.

5) If $\frac{3}{2\sqrt{2}}<\mid\beta\mid<\frac{2}{\sqrt{3}}$ then
$\frac{3\pi}{\mid\beta\mid}< T<2\pi\left(\mid\beta\mid+\sqrt{\beta^2-1}\right)< \frac{4\pi}{\mid\beta\mid}$
and  $T$ satisfies the system of equations
$$\cos kx=\frac{1}{\sqrt{1+k^2\Tg^2x}},\quad
\sin kx=\frac{k\Tg x}{\sqrt{1+k^2\Tg^2 x}}<0,$$
where
\begin{equation}
\label{xk}
k=\frac{\mid\beta\mid}{\sqrt{\beta^2-1}}, \quad x=\frac{T\sqrt{\beta^2-1}}{2}=\frac{T}{2\sqrt{k^2-1}}.
\end{equation}

6) If $1<\mid\beta\mid<\frac{3}{2\sqrt{2}}$ then
$\frac{2\pi}{\mid\beta\mid}<2\pi\left(\mid\beta\mid+\sqrt{\beta^2-1}\right)<T<\frac{3\pi}{\mid\beta\mid}$
and $T$ satisfies the system of equations
$$\cos kx=\frac{-1}{\sqrt{1+k^2\Tg^2x}},\quad
\sin kx=\frac{-k\Tg x}{\sqrt{1+k^2\Tg^2 x}}<0,$$
where $k$ and $x$ are defined by formulae (\ref{xk}).

\end{theorem}

\begin{proof}

Let $\beta\neq 0$ and $\gamma=\gamma(\beta,\phi;t)$, $0\leq t\leq T$, be a noncontinuable shortest arc (\ref{geod}). It follows from Definition \ref{cutl} and Theorem  \ref{cutloc} that $c:=\gamma(T)$ belongs to $C(e)=K(e)\cup S_1(e)$.

Assume that $c\in S_1(e)$. Then $m(c)=0$ on the ground of (\ref{congloc}) and Proposition \ref{so2}.
Therefore $\mid\beta\mid>1$ and $T=\frac{2\pi}{\sqrt{\beta^2-1}}$ by (\ref{mn1}), (\ref{mn2}), (\ref{mn3}). Additionally,
$|\beta|$ is the largest number, for which the right hand side in  (\ref{ss})
is equal to $c,$ and function
$$\xi(\beta)=\frac{\pi\beta}{\sqrt{\beta^2-1}},\quad \beta\in \left[\frac{2}{\sqrt{3}},+\infty\right),$$
is monotonically decreases from $2\pi$ to $\pi.$ In consequence of this and (\ref{ss}), (\ref{congloc}),
to the element $c=-e$ (to the set $S_1(e)$) corresponds $\beta=\pm \frac{2}{\sqrt{3}}$
(the set $\{\beta: \mid\beta\mid\geq\frac{2}{\sqrt{3}}\}).$ Conversely, if $T=\frac{2\pi}{\sqrt{\beta^2-1}}$ and 
$\mid\beta\mid\geq\frac{2}{\sqrt{3}}$ then $c\in S_1(e)$ by (\ref{ss}). Item~1) is proved.

Assume that $c\in K(e)$. Then by (\ref{cutloc1}) and (\ref{sistem}),
\begin{equation}
\label{equiv}
n\sin\frac{\beta T}{2}-\beta m\cos\frac{\beta T}{2}=0,\quad n\cos\frac{\beta T}{2}+\beta m\sin\frac{\beta T}{2}<0.
\end{equation}

Let $\mid\beta\mid=1$. In view of (\ref{mn1}), the conditions (\ref{equiv}) can be written in the form
$$\Tg{\frac{T}{2}}=\frac{T}{2},\quad \cos\frac{T}{2}<0.$$
Then $\pi<\frac{T}{2}<\frac{3\pi}{2}$, i.e. $T\in (2\pi,3\pi)$. Therefore,
$$\cos{\frac{T}{2}}=\frac{-1}{\sqrt{1+\Tg^2{\frac{T}{2}}}}=\frac{-1}{\sqrt{1+(T/2)^2}},\quad
\sin{\frac{T}{2}}=\Tg{\frac{T}{2}}\cdot\cos{\frac{T}{2}}=
\frac{-T/2}{\sqrt{1+(T/2)^2}}$$
and item~2) of Theorem \ref{mai} is proved.

Let $0<\mid\beta\mid<1$. In view of  (\ref{mn2}), the conditions (\ref{equiv}) can be written in the form
\begin{equation}
\label{t}
\Tg{\left(\frac{\mid\beta\mid T}{2}\right)}=\frac{\mid\beta\mid}{\sqrt{1-\beta^2}}
\Th{\left(\frac{T\sqrt{1-\beta^2}}{2}\right)},\quad \cos{\frac{\mid\beta\mid T}{2}}<0.
\end{equation}
Let us use the notation (\ref{kx}). Then $x>0$, $k>0$, and conditions  (\ref{t}) take the form
\begin{equation}
\label{tr}
f(x)=f(k,x):=\Tg{kx}-k\Th{x}=0,\quad \cos{kx}<0.
\end{equation}

Let us fix $k>0$. Note that the function $f(x)$ increases because
\begin{equation}
\label{fpr}
f^{\prime}(x)=k\left(\frac{1}{\cos^2kx}-\frac{1}{\Ch^2x}\right)>0.
\end{equation}
Besides,
$$\lim\limits_{x\rightarrow 0}f(x)=0,\quad \lim\limits_{x\rightarrow\frac{\pi}{2k}-0}f(x)=+\infty,\quad
\lim\limits_{x\rightarrow\frac{\pi}{2k}+0}f(x)=-\infty,$$
$$f\left(\frac{\pi}{k}\right)=-k\Th\frac{\pi}{k}<0,\quad
\lim\limits_{x\rightarrow\frac{3\pi}{2k}-0}f(x)=+\infty.$$
Therefore $f(x)>0$ for any $x\in\left(0,\frac{\pi}{2k}\right)$, $f(x)<0$ for any $x\in\left(\frac{\pi}{2 k},\frac{\pi}{k}\right]$ and $f(x)$ has a unique zero on interval $\left(\frac{\pi}{k},\frac{3\pi}{2k}\right)$. It follows from this and (\ref{kx}) that
$$\frac{\pi\sqrt{1-\beta^2}}{\mid\beta\mid}<\frac{T\sqrt{1-\beta^2}}{2}<\frac{3\pi\sqrt{1-\beta^2}}{2\mid\beta\mid},$$ i.e.
$T\in\left(\frac{2\pi}{\mid\beta\mid},\frac{3\pi}{\mid\beta\mid}\right)$. Now by (\ref{tr}) and inclusion $kx\in (\pi,3\pi/2)$,
$$\cos{kx}=\frac{-1}{\sqrt{1+\Tg^2{kx}}}=\frac{-1}{\sqrt{1+k^2\Th^2x}},$$
$$\sin{kx}=\Tg{kx}\cdot\cos{kx}=\frac{-k\Th x}{\sqrt{1+k^2\Th^2x}}.$$
Item~3) of Theorem \ref{mai} is proved.
It remains to consider the case $1<\mid\beta\mid<\frac{2}{\sqrt{3}}$.

At first assume that  $n=0$. It follows from (\ref{mn3}), (\ref{equiv}), and Proposition \ref{sss} that
\begin{equation}
\label{nnt}
T=\frac{\pi}{\sqrt{\beta^2-1}},\quad\cos\frac{\beta T}{2}=0.
\end{equation}
This implies that
\begin{equation}
\label{beta}
\frac{\mid\beta\mid T}{2}=\frac{k\pi}{2}=\frac{\pi}{2}+\pi l,\quad \frac{\mid\beta\mid}{\sqrt{\beta^2-1}}=k=1+2l,\quad
l\in\mathbb{N}.
\end{equation}

If $l=1$ then $k=3,$ $\mid\beta\mid=\frac{3}{2\sqrt{2}}$, $T=2\sqrt{2}\pi$.

It follows from (\ref{beta}) that $k>3,$ $1<\mid\beta\mid<\frac{3}{2\sqrt{2}},$ and $T>\frac{3\pi}{\mid\beta\mid}$ if $l>1.$

Now assume that $n\neq 0$. Then in view of (\ref{mn3}), the equality (\ref{equiv}) can be written in the form
\begin{equation}
\label{t1}
\Tg{\left(\frac{\mid\beta\mid T}{2}\right)}=\frac{\mid\beta\mid}{\sqrt{\beta^2-1}}
\Tg{\left(\frac{T\sqrt{\beta^2-1}}{2}\right)}.
\end{equation}
Let us use the notation (\ref{xk}). It follows from (\ref{mn3}) and Proposition \ref{sss} that
\begin{equation}
\label{ineq}
x\in\left(0,\frac{\pi}{2}\right)\cup\left(\frac{\pi}{2},\pi\right),\quad k>2,
\end{equation}
and the equality (\ref{t1}) takes the form
\begin{equation}
\label{t0}
g(x)=g(k,x):=\Tg{kx}-k\Tg{x}=0.
\end{equation}
One can easily see that in view of (\ref{mn3}), (\ref{t0}), and our notation, the inequality in (\ref{equiv}) is equivalent to condition $\cos x\cos kx<0$.

Let us fix $k>2$ and consider the function $g(x)$ on 
$\left(0,\frac{\pi}{2}\right)\cup\left(\frac{\pi}{2},\pi\right)$.
Note that
\begin{equation}
\label{qprime}
g^{\prime}(x)=\frac{k}{\cos^2 kx}-\frac{k}{\cos^2 x}=k\left(\Tg^2 kx-\Tg^2 x\right)=\frac{k\sin[(k-1)x]\sin[(k+1)x]}{\cos^2 x\cos^2 kx}.
\end{equation}

Let $k=3,$ which is equivalent to equality $\mid\beta\mid=\frac{3}{2\sqrt{2}}$. Using the tangent of the triple argument formula, we get
$$g(x)=\Tg 3x-3\Tg x=\frac{3\Tg x-\Tg^3 x}{1-3\Tg^2 x}-3\Tg x=\frac{8\Tg^3 x}{1-3\Tg^2 x}.$$
Therefore for $k=3$ the equation (\ref{t0}) has no solution on 
$\left(0,\frac{\pi}{2}\right)\cup\left(\frac{\pi}{2},\pi\right)$. Together with the above-mentioned case $n=0,$ this fact proves item~4) of Theorem \ref{mai}.

Let $2<k<3,$ i.e. $\frac{3}{2\sqrt{2}}<\mid\beta\mid<\frac{2}{\sqrt{3}}$. One can easily see that  $x\neq\frac{\pi}{2k}$, $x\neq\frac{3\pi}{2k}$ and the equation $g^{\prime}(x)=0$  has four different solutions $x=\frac{\pi}{k+1}$, $x=\frac{2\pi}{k+1}$, $x=\frac{3\pi}{k+1},$ and $x=\frac{\pi}{k-1}$ on 
 $\left(0,\frac{\pi}{2}\right)\cup\left(\frac{\pi}{2},\pi\right)$, while
$$0<\frac{\pi}{2k}<\frac{\pi}{k+1}<\frac{\pi}{2}<\frac{2\pi}{k+1}<\frac{3\pi}{2k}<\frac{\pi}{k-1}<\frac{3\pi}{k+1}<\pi.$$
Note that
\begin{equation}
\label{nn1}
\lim\limits_{x\rightarrow 0}g(x)=0,\quad
\lim\limits_{x\rightarrow\frac{\pi}{2k}-0}g(x)=+\infty,\quad
\lim\limits_{x\rightarrow\frac{\pi}{2k}+0}g(x)=-\infty,
\end{equation}
$$g\left(\frac{\pi}{k+1}\right)=-(k+1)\Tg\frac{\pi}{k+1}<0,\quad
\lim\limits_{x\rightarrow\frac{\pi}{2}-0}g(x)=-\infty,\quad
\lim\limits_{x\rightarrow\frac{\pi}{2}+0}g(x)=+\infty,$$
$$g\left(\frac{2\pi}{k+1}\right)=-(k+1)\Tg\frac{2\pi}{k+1}>0,\quad
\lim\limits_{x\rightarrow\frac{3\pi}{2k}-0}g(x)=+\infty,$$
$$\lim\limits_{x\rightarrow\frac{3\pi}{2k}+0}g(x)=-\infty,\quad
g\left(\frac{\pi}{k-1}\right)=(1-k)\Tg\frac{\pi}{k-1}>0.$$
It follows from this and (\ref{qprime}) that  for $2<k<3$ the equation (\ref{t0}) is solvable on 
$\left(0,\frac{\pi}{2}\right)\cup\left(\frac{\pi}{2},\pi\right)$, and the smallest root of this equation belongs to interval
$\left(\frac{3\pi}{2k},\frac{\pi}{k-1}\right).$ In view of (\ref{xk}), this means that
$$\frac{3}{2\sqrt{2}}<\mid\beta\mid<\frac{2}{\sqrt{3}},\quad \frac{3\pi}{\mid\beta\mid}<T<2\pi\left(\mid\beta\mid+\sqrt{\beta^2-1}\right)<\frac{4\pi}{\mid\beta\mid}.$$
Then $\cos x<0$, $\Tg x < 0,$ $\cos kx>0,$  and on the ground of (\ref{t0}),
$$\cos kx=\frac{1}{\sqrt{1+\Tg^2 kx}}=\frac{1}{\sqrt{1+k^2\Tg^2 x}},$$
$$\sin kx=\Tg kx\cdot\cos kx=\frac{k\Tg x}{\sqrt{1+k^2\Tg^2 x}}< 0.$$
Item~5) of Theorem \ref{mai} is proved.

Let $k>3,$ i.e. $1<\mid\beta\mid<\frac{3}{2\sqrt{2}}$. One can easily see that $x\neq\frac{\pi}{2k}$, $x\neq\frac{3\pi}{2k}$ and the equation  $g^{\prime}(x)=0$ has on interval $\left(0,\frac{\pi}{2}\right)$ at least five  different solutions $x=\frac{\pi}{k+1}$, $x=\frac{\pi}{k-1}$, $x=\frac{2\pi}{k+1}$, $x=\frac{2\pi}{k-1}$, $x=\frac{3\pi}{k+1}$, while
$$0<\frac{\pi}{2k}<\frac{\pi}{k+1}<\frac{\pi}{k-1}<\frac{3\pi}{2k}<\frac{2\pi}{k+1}<\frac{\pi}{2}.$$
Note that the equalities (\ref{nn1}) are fulfilled and
$$g\left(\frac{\pi}{k+1}\right)=-(k+1)\Tg\frac{\pi}{k+1}<0,\quad
g\left(\frac{\pi}{k-1}\right)=(1-k)\Tg\frac{\pi}{k-1}<0,$$
$$\lim\limits_{x\rightarrow\frac{3\pi}{2k}-0}g(x)=+\infty.$$
It follows from this and (\ref{qprime}) that for $k>3$ the equation (\ref{t0}) is solvable on $\left(0,\frac{\pi}{2}\right)\cup\left(\frac{\pi}{2},\pi\right)$, and the smallest root of this equation  belongs to interval
$\left(\frac{\pi}{k-1},\frac{3\pi}{2k}\right).$ Since we proved previously that $T>\frac{3\pi}{\mid\beta\mid}$ for $n=0,$ then
$$\frac{2\pi}{\mid\beta\mid}<2\pi\left(\mid\beta\mid+\sqrt{\beta^2-1}\right)<T<\frac{3\pi}{\mid\beta\mid}.$$
Hence $\cos x>0$, $\cos kx<0$ and on the ground of (\ref{t0}),
$$\cos kx=\frac{-1}{\sqrt{1+\Tg^2 kx}}=\frac{-1}{\sqrt{1+k^2\Tg^2 x}},$$
$$\sin kx=\Tg kx\cdot\cos kx=\frac{-k\Tg x}{\sqrt{1+k^2\Tg^2 x}}< 0.$$
Item~6) of Theorem \ref{mai} is proved.
\end{proof}

\begin{theorem}
\label{mai2}
Let $\gamma=\gamma(\beta,\phi; t)$, $0\leq t\leq T$, be a noncontinuable shortest arc (\ref{geod}). Then

1) Function $T=T(\mid\beta\mid)$ strictly decreases on intervals
 $\left(0,\frac{3}{2\sqrt{2}}\right]$, $\left[\frac{2}{\sqrt{3}},+\infty\right)$ and strictly increases on segment
 $\left[\frac{3}{2\sqrt{2}},\frac{2}{\sqrt{3}}\right]$.

2) Function $T=T(\mid\beta\mid)$ is continuous, piecewise real--analytic and
$T(0,+\infty)=(0,+\infty)$.

3) Function $T=T(\mid\beta\mid)$ has local minimum $2\sqrt{2}\pi$ at
$3/2\sqrt{2}$ and  local maximum $2\sqrt{3}\pi$ at $2/\sqrt{3}.$
\end{theorem}

\begin{proof}

1) In consequence of item~1 of Theorem \ref{mai}, it is true for $\mid\beta\mid\geq\frac{2}{\sqrt{3}}.$

If $0<\mid\beta\mid<1$ then the relations (\ref{tr}) and (\ref{kx}) define implicit function $T=T(k)$. Additionally, in consequence of (\ref{fpr}), (\ref{kx}), and item 3 of Theorem \ref{mai},
$$f'_T=f'_x\cdot x'_T=\frac{k}{2\sqrt{1+k^2}}\left(\frac{1}{\cos^2kx}-\frac{1}{\Ch^2x}\right)>0,$$
$$f'_k=\frac{x}{\cos^2kx}-\Th x -\frac{kT}{2}\left(\frac{1}{\cos^2kx}-\frac{1}{\Ch^2x}\right)\frac{k}{(1+k^2)^{3/2}}=$$
$$x(1+k^2\Th^2x)-\Th x -\frac{xk^2}{1+k^2}(1+k^2\Th^2x-(1-\Th^2x))=x-\Th x >0.$$
Then by the theorem about the derivative of implicit function,
$\frac{dT}{d k}=-\frac{f^{\prime}_{k}}{f^{\prime}_{T}}<0;$
since $k'_{\beta}=1/(1-\beta^2)^{3/2}>0,$ then $T=T(|\beta|)$ strictly decreases for
$\beta^2 < 1.$

Let $1<\mid\beta\mid<\frac{2}{\sqrt{3}}$ and
$\mid\beta\mid \neq \frac{3}{2\sqrt{2}}$. Then the relations (\ref{t0}) and (\ref{xk}) define implicit function $T=T(k).$ Additionally, in consequence of  (\ref{qprime}), (\ref{xk}), items 5), 6) of Theorem \ref{mai},
$$g'_T=g'_x\cdot x'_T=\frac{k}{2\sqrt{k^2-1}}\left(\frac{1}{\cos^2kx}-\frac{1}{\cos^2x}\right)=$$
$$\frac{k}{2\sqrt{k^2-1}}(1+k^2\Tg^2x-(1+\Tg^2x))=\frac{k\sqrt{k^2-1}\Tg^2x}{2} > 0,$$
$$g'_k=\frac{x}{\cos^2kx}-\Tg x -\frac{kT}{2}\left(\frac{1}{\cos^2kx}-\frac{1}{\cos^2x}\right)\frac{k}{(k^2-1)^{3/2}}=$$
$$x(1+k^2\Tg^2x)-\Tg x - x\frac{k^2}{(k^2-1)}(1+k^2\Tg^2x-(1+\Tg^2x))=x-\Tg x.$$

It follows from the proof of Theorem \ref{mai} that
$x\in\left(0,\frac{\pi}{2}\right)$ for $1<\mid\beta\mid<\frac{3}{2\sqrt{2}}$ and
$x\in\left(\frac{\pi}{2},\pi\right)$ for $\frac{3}{2\sqrt{2}}<\mid\beta\mid<\frac{2}{\sqrt{3}}$.
Consequently, by the theorem about the derivative of  implicit function,
$\frac{dT}{d k}=-\frac{g^{\prime}_{k}}{g^{\prime}_{T}}$ is positive
for $1<\mid\beta\mid<\frac{3}{2\sqrt{2}}$ and negative for
$\frac{3}{2\sqrt{2}}<\mid\beta\mid<\frac{2}{\sqrt{3}}$.
Since $k'_{\beta}=-1/(\beta^2-1)^{3/2}< 0,$ then $T=T(|\beta|)$ strictly decreases for
$1<\mid\beta\mid<\frac{3}{2\sqrt{2}}$ and strictly increases for
$\frac{3}{2\sqrt{2}}<\mid\beta\mid<\frac{2}{\sqrt{3}}$.

2) The required statement is easy to verify on the ground of the proof of item 1) of Theorem \ref{mai2} and 
Theorem \ref{mai}.

3) The statement follows from items 1, 2 of Theorem \ref{mai2} and items 1), 4) of Theorem \ref{mai}.
\end{proof}

\end{document}